\documentclass[11pt,a4paper]{article}
\usepackage[dvips]{graphicx}
\usepackage{amsmath}
\usepackage{amsfonts}
\usepackage{amsthm}
\usepackage{amssymb}
\usepackage{mathrsfs}
\usepackage[T1]{fontenc}
\usepackage[latin1]{inputenc}
\usepackage[all]{xy}
\usepackage{a4wide}
\usepackage{psfrag}
\usepackage[bf]{caption}

\newcommand{\tykkp}{\mathbb{P}}
\newcommand{\fino}{\mathcal{O}}

\newcommand{\finixi}{\mathcal{I}_{\xi}}

\newcommand{\finfca}{\mathcal{F}_{C,A}}

\newcommand{\fing}{\mathcal{G}}
\newcommand{\fini}{\mathcal{I}}

\newcommand{\rk}{\mathrm{rk}}
\newtheorem{Prop}{Proposition}[section]

\newtheorem{Theorem}[Prop]{Theorem}
\newtheorem{Lemma}[Prop]{Lemma}
\newtheorem{Cor}[Prop]{Corollary}

\newtheoremstyle{Property}{\topsep}{\topsep}%
  {}
  {}
  {\bfseries}
  {.}
  { }
  {\thmname{#1}\thmnumber{ #2}\thmnote{ #3}}

\theoremstyle{Property}

\newtheoremstyle{Eksempel}{\topsep}{\topsep}%
  {}
  {}
  {\bfseries}
  {.}
  { }
  {\thmname{#1}\thmnumber{ #2}\thmnote{ #3}}

\theoremstyle{Example}

\newtheoremstyle{Remark}{\topsep}{\topsep}%
  {}
  {}
  {\bfseries}
  {.}
  { }
  {\thmname{#1}\thmnumber{ #2}\thmnote{ #3}}

\theoremstyle{Remark}
\newtheorem{Remark}[Prop]{Remark}

\AtEndDocument{\bigskip{\footnotesize%
  \textsc{Department of Teacher Education, Telemark University College, Norway}  \par
  \textit{E-mail address}, \texttt{nils.h.rasmussen@hit.no} \par
}}

\title{Pencils and nets of small degree on curves on smooth, projective surfaces of Picard rank 1 and very ample generator}
\author{Nils Henry Rasmussen}
\date\today
\begin{document}

\maketitle

\begin{abstract}
Let $S$ be a smooth, projective surface of Picard rank 1 and very ample generator embedding $S$ into $\tykkp^n$. Let $C\in |\fino_S(m)|$ for $m\geq 5$ be a smooth curve. We prove that any base-point free, complete $g^r_d$ on $C$ for $r\in\{1,2\}$ and $d$ small enough is cut out by a hyperplane section restricted to a multisecant $(n-r-1)$-plane.
\end{abstract}

\section{Introduction}\label{intro}

Let $C$ be a smooth curve embedded in $\tykkp^n$, and let $|A|$ be a complete, base-point free linear series $g^r_d$ on $C$, where $r \in \{1, 2\}$. Our aim in this paper is to find instances for when $|A|$ is cut out by hyperplane sections restricted to an $(n-r-1)$-plane secant on $C$.

The question in focus has previously been much studied, in particular for the case when $C \subseteq \tykkp^3$. In \cite{ciliberto1984uniqueness}, it was proven that when $C$ lies on a surface of degree $>4$ in $\tykkp^3$, the linear system of plane sections is the only $g^3_{\deg(C)}$ that exists. In \cite{basili1996indice}, it was proven that when $C$ is a smooth complete intersection curve in $\tykkp^3$, any pencil computing the gonality is given by a plane section minus a multisecant of maximal degree. In \cite{Hartshorne08}, Hartshorne and Schlesinger proved that the gonality of a general ACM curve in $\tykkp^3$ is also computed by multisecants. A further overview of previous research done on gonality and multisecants is given in the introduction of \cite{Hartshorne08}.

This paper is an attempt to prove a slightly more general result for the case where $C$ lies on a smooth, projective surface $S$ with Picard rank $1$ and very ample generator. We will here study $g^1_d$'s of not necessarily minimal degree, in addition to $g^2_d$'s. The methods we use are taken from \cite[Proposition 5.5]{Hartshorne08}. The idea is to create a map from $\fino_S(1)$ to the $g^r_d$ $A$ on $C$, and make sure that the map is surjective on global sections.

In what follows, $S$ will denote a smooth, projective surface with Picard group generated by a very ample line bundle $\fino_S(H)$ satisfying $h^0(\fino_S(H))\geq 4$. This surface could be a K3 surface with Picard rank 1, embedded into projective space using $\fino_S(H)$, or a general complete-intersection surface in $\tykkp^n$ for $n\geq 3$, except if it is of degree $\leq 3$ in $\tykkp^3$ or type $(2,2)$ in $\tykkp^4$ (by a generalisation of the Noether--Lefschetz theorem, see \cite{lefschetz1921certain}. Also, the introduction in \cite{kim1991noether} gives a nice overview of further research done in this area).

We will always be considering $S$ as embedded into projective space by $\fino_S(H)$. When writing $\fino_S(m)$, we mean $\fino_S(mH)$. The characteristic is assumed to be $0$ throughout the paper.

Our main result is the following:

\begin{Theorem}\label{Main}
Let $S$ and $\fino_S(H)$ be as above. Let $C\in |\fino_S(m)|$ for $m\geq 5$ be a smooth curve on $S$. Then any base-point free, complete $g^r_d$ on $C$, with $r\in\{1,2\}$ and $d\leq mH^2$, is given by restricting hyperplane sections to an $(n-r-1)$-plane secant on $C$.
\end{Theorem}

An immediate consequence of Theorem \ref{Main} is the following:

\begin{Cor}\label{cor}
Let $S$ be a general complete intersection surface in $\tykkp^4$ of type $(2,3)$ (hence a K3 surface), with Picard group generated by the class of hyperplane sections $H$. Then, for $m\geq 5$, the minimal degree of a $g^2_d$ on a smooth curve $C\in|\fino_S(m)|$ is $6m-3$.
\end{Cor}

\begin{proof}
Let $T$ be the Fano scheme of lines contained in the quadric hypersurface containing $S$. This scheme corresponds precisely to the set of $3$-secant lines on $S$. The dimension of $T$ is $3$.

Let $\mathcal{I}$ be the incidence variety given by pairs $(C,\Gamma)$ such that $C$ is smooth in $|\fino_S(m)|$ and $\Gamma\in T$ is a $3$-secant line on $C$. Denoting smooth curves in $|\fino_S(m)|$ by $|\fino_S(m)|_s$, we have the following diagram, where $p$ and $q$ denote the natural projections:
$$\xymatrix{\mathcal{I} \ar[d]^<<<<<{q} \ar[r]^<<<<<{p} & |\fino_S(m)|_s\\
T&\,\,\,\,\,\,\,\,\,\,\,\,\,\,\,\,\,\,\,\,\,\,\,\,\,\,\,\,\,.}$$

For any $\Gamma\in T$, $\Gamma\cap S$ will impose $3$ independent conditions on the elements in $|\fino_S(m)|_s$. Since the elements in $|\fino_S(m)|_s$ are general in $|\fino_S(m)|$, giving us that $\dim|\fino_S(m)|_s=\dim|\fino_S(m)|=3m^2+1$, it is therefore clear that $q$ is surjective, and that each fibre corresponds to a codimension $3$ subscheme of $|\fino_S(m)|_s$. Since $\dim(T)=3$, it follows that each fibre of $p$ must have dimension $0$, and so each curve in $|\fino_S(m)|_s$ has a finite number of $3$-secant lines.

Since $S$ can't have any $4$-secant lines, it follows that the minimal degree of a $g^2_d$ on any $C\in|\fino_S(m)|_s$ is $C.H-3=mH^2-3=6m-3$.
\end{proof}

\section{Proof of the theorem}
Following the work of Lazarsfeld and Tyurin \cite{Lazarsfeld, Tyurin}, given a smooth curve $C$ of genus $g$ on $S$ and a base-point free, complete $g^r_d$ $|A|$ on $C$, one defines a vector-bundle $\finfca$ on $S$ as the kernel of the evaluation morphism $H^0(C,A)\otimes\fino_S\rightarrow A \rightarrow 0$. The bundle $\finfca$ (or, more frequently, its dual) is often referred to as the \emph{associated Lazarsfeld--Mukai bundle} of $C$ and $A$. The bundle has the following properties:
\begin{itemize}
\item $\mathrm{rk}(\finfca)=r+1$.
\item $\mathrm{det}(\finfca)=\fino_S(-C)$.
\item $c_2(\finfca)=d$.
\item $h^0(S,\finfca)=0$.
\item The dual, $\finfca^{\vee}$, is globally generated away from a finite set.
\end{itemize}

In the cases we are interested in, where $C\in |\fino_S(m)|$ with $m\geq 5$, and $d\leq mH^2$, it follows from \cite[Theorem]{bogomolov1995stable} that $\finfca$ is non-$\fino_S(H)$-stable. There thus exists a maximal destabilising sequence,
\begin{equation}\label{exact}
0\rightarrow M\rightarrow\finfca\rightarrow N\rightarrow 0,
\end{equation}
where $M$ is a vector-bundle with $\rk(M)\leq r$ satisfying $\frac{1}{\mathrm{rk}(M)}c_1(M).H>-\frac{1}{r+1}mH^2$, and $N$ is torsion-free and $\fino_S(H)$-stable.

Before presenting the proof of Theorem \ref{Main}, we prove the following lemma:

\begin{Lemma}\label{lemma}
Let $\finfca$ be as above, with $C\in |\fino_S(m)|$, $m\geq 5$, $d\leq mH^2$ and $r\leq 2$. In the exact sequence \eqref{exact}, we have the following:
\begin{itemize}
\item $c_1(M)=\fino_S(-H)$,
\item $c_2(M)\geq 0$, and
\item $c_2(N)\geq 0$.
\end{itemize}
As a consequence of the above, $\rk(M)=r$ and $\rk(N)=1$.
\end{Lemma}

\begin{proof}
We begin by proving that $c_1(M)=\fino_S(-a)$ for some $a>0$. In the case where $\mathrm{rk}(M)=1$, this is clear since $h^0(S,\finfca)=0$. In the case where $\mathrm{rk}(M)=2$, we dualise \eqref{exact} and get
$$0\rightarrow N^{\vee}\rightarrow\finfca^{\vee}\rightarrow \tilde{M}\rightarrow 0,$$
where $\tilde{M}$ satisfies $\tilde{M}^{\vee}=M$. Since $\finfca^{\vee}$ is globally generated away from a finite set, then the same must be the case for $\tilde{M}$, and so $c_1(\tilde{M})=\fino_S(a)$ with $a\geq 0$ (proof: there exists a saturated sub-linebundle $\fino_S(t)$ with $t\geq 0$ which we can inject into $\tilde{M}$, and the cokernel must be globally generated away from a finite set). If $a=0$, then $\tilde{M}=\fino_S^{\oplus 2}$. However, it then follows that $c_2(\finfca)=0$, which is a contradiction. It follows that $a>0$.

We now prove that $c_2(M)$ and $c_2(N)$ are both nonnegative. Since $N$ is torsion-free and $\fino_S(H)$-stable, then by \cite[Theorem]{bogomolov1995stable}, $c_2(N)\geq 0$. The inequality $c_2(M)\geq 0$ is clear for the $\rk(M)=1$ case. For the $\rk(M)=2$ case, suppose $c_2(M)<0$, and note that by \cite[Theorem]{bogomolov1995stable}, $M$ must then be non-$\fino_S(H)$-stable. This gives us a maximal destabilising sequence
$$0\rightarrow \fino_S(-b)\rightarrow M\rightarrow \fino_S(b-a)\otimes\fini_{\eta}\rightarrow 0,$$
where $b>0$ since $\finfca$ and hence also $M$ cannot have any global sections; where $\fini_{\eta}$ is the ideal sheaf of a finite subscheme $\eta$ (possibly empty); and where $-bH^2>-\frac{1}{2}aH^2$, implying that $2b<a$. Since $c_2(M)<0$, we have $(-b)(b-a)H^2<0$, implying that $b-a>0$, giving us the desired contradiction.

We now prove that $a=1$: If $c_1(M)=\fino_S(-a)$ and $c_1(N)=\fino_S(a-m)$ for $a\geq 2$ satisfying $\frac{1}{\mathrm{rk}(M)}c_1(M).H> -\frac{1}{r+1}mH^2$, this gives us $c_2(\finfca)\geq c_1(M).c_1(N)=a(m-a)H^2> mH^2$, which contradicts the assumption we have for $c_2(\finfca)$. (As long as $m\geq 3$, we cannot have $c_1(M)=\fino_S(1-m)$ at the same time as $\frac{1}{\mathrm{rk}(M)}c_1(M).H> -\frac{1}{r+1}mH^2$. The condition $m\geq 5$ is to ensure that the inequality $a(m-a)H^2>mH^2$ is strict.)

It now follows that $\rk(M)=r$, because of the following: If $r=2$ with $\rk(M)=1$, then by \cite[Theorem]{bogomolov1995stable}, the stability of $N$ implies that $\frac{1}{4}c_1(N)^2-c_2(N)\leq 0$, so that $c_2(N)\geq\frac{1}{4}(m-1)^2H^2$. This gives us $c_2(\finfca)\geq (m-1)H^2+\frac{1}{4}(m-1)^2H^2>mH^2$, a contradiction.
\end{proof}

We now prove the main theorem.

\begin{proof}[Proof of Theorem \ref{Main}] Let $C\in |\fino_S(m)|$ be smooth, where $m\geq 5$, let $A$ be a base-point free, complete $g^r_d$ on $C$ with $r\in\{1, 2\}$ and $d\leq mH^2$, and let $\finfca$ be the associated Lazarsfeld--Mukai vector bundle. By Lemma \ref{lemma}, there exists a rank $r$ vector bundle $M$ with $c_1(M)=\fino_S(-H)$ and $c_2(M),c_2(N)\geq 0$ such that
$$0\rightarrow M\rightarrow\finfca\rightarrow N\rightarrow 0.$$
Because $M$ injects into $\finfca$, we can compose with the map from $\finfca$ into $\fino_S^{\oplus r+1}$ and get the following commutative diagram, where $\fing$ is the cokernel:

\begin{equation}\nonumber
\xymatrix{0 \ar[r] & M \ar[r]\ar[d] & \fino_S^{\oplus r+1} \ar[r]^<<<<<{\widetilde{\mathrm{ev}}}\ar@{=}[d] & \fing \ar[r]\ar[d]^{\phi} & 0 \\
            0 \ar[r] & \finfca \ar[r] & \fino_S^{\oplus r+1} \ar[r]^<<<<<{\mathrm{ev}} & A \ar[r] & 0.
}
\end{equation}
We see that $\fing$ is torsion-free, since any possible torsion element of $\fing$ would map to $0$ in $A$, and by the snake lemma, $\ker(\phi)\cong N$, which is torsion-free. We have $\rk(\fing)=1$ and $c_1(\fing)=\fino_S(1)$.

Note that $h^0(N)=0$, since $N$ is torsion-free of rank $1$ and has negative $c_1$. This implies that the image of each global section under the map $\phi$ must be nonzero in $A$. Furthermore, since $\mathrm{ev}:\fino_S^{\oplus r+1}\rightarrow A$ is surjective on global sections, the same must apply for $\phi:\fing\rightarrow A$. It follows that $h^0(S,\fing)=h^0(C,A)=r+1$.

Since $\mathrm{rk}(\fing)=1$, then $\fing=\fino_S(1)\otimes\fini_{\xi}$, where $\finixi$ is the ideal sheaf of a finite subscheme $\xi$ (nonempty because $h^0(S,\fino_S(1))>r+1$). Let $A'\in |A|$. We prove that $A'+\xi'=H'_{|C}$, where $\xi':=\xi\cap C$ and $H'$ is a hyperplane section. The map $\phi$ corresponds to a nonzero element in $\mathrm{Hom}_{\fino_S}(\fino_S(1)\otimes\fini_{\xi},A)=\mathrm{Hom}_{\fino_S}(\fini_{\xi},A\otimes\fino_S(-1))$. This implies that $h^0(S-\xi,A\otimes\fino_S(-1))>0$, i.e., $A'>H'_{|C}-\xi'$ for some hyperplane section $H'$. Since $A$ is base-point free, then this means that either $A'=H'_{|C}-\xi'$ or $h^0(C,A)>h^0(C,\fino_C(1)\otimes\mathcal{I}_{\xi'})$. However, since $h^0(C,A)=h^0(S,\fino_S(1)\otimes\finixi)$, the latter inequality would imply that $\fino_S(1)\rightarrow\fino_C(1)$ has a nonzero kernel on global sections, which cannot be the case since we are assuming that $C\in\fino_S(m)$ with $m\geq 5$. It thus follows that $A'=H'_{|C}-\xi'$. Since $|A|$ is a $g^r_d$, then the base-locus of the hyperplane sections containing $\xi'$ must be an $(n-r-1)$-plane multisecant.
\end{proof}

\begin{Remark}
It would be interesting to see if this result can be extended to $3\leq r \leq H^0(\fino_S(1))-2$. Most of the methods used in Lemma \ref{lemma} seem to work, but with the excpetion that we lose control over $c_2(M)$. However, if we manage to prove that $c_2(M)\geq 0$ in some (or all) of these cases, then Theorem \ref{Main} would also apply for these values of $r$.
\end{Remark}

\subsection*{Acknowledgments}
Thanks to Andreas Leopold Knutsen, Ciro Ciliberto and Shengtian Zhou for helpful remarks in the process of this work. The author would also like to thank the reviewer for helpful comments.

\bibliographystyle{plain}
\bibliography{NHR}

\end{document}